\documentclass[a4paper,oneside]{amsart}
\usepackage{geometry}

\usepackage{mathtext}
\usepackage[T1]{fontenc}
\usepackage{color}
\usepackage{amsmath}
\usepackage{amsfonts,amssymb,amsthm}
\usepackage[pdftex]{graphicx}
\usepackage{comment}
\usepackage{float}
\usepackage{todonotes}
\usepackage{microtype}

\usepackage[colorlinks]{hyperref}
\hypersetup{
bookmarksnumbered,
pdfstartview={FitH},
breaklinks=true,
linkcolor=blue,
urlcolor=blue,
citecolor=blue,
bookmarksdepth=2
}

\usepackage{tikz}
\usetikzlibrary{positioning, arrows.meta}
\usetikzlibrary{decorations.markings}
\tikzset{middlearrow/.style={
        decoration={markings,
            mark= at position 0.5 with {\arrow{#1}} ,
        },
        postaction={decorate}
    }
}
\tikzset{firstthirdarrow/.style={
        decoration={markings,
            mark= at position 0.33 with {\arrow{#1}} ,
        },
        postaction={decorate}
    }
}
\tikzset{secondthirdarrow/.style={
        decoration={markings,
            mark= at position 0.66 with {\arrow{#1}} ,
        },
        postaction={decorate}
    }
}
\usetikzlibrary{cd}

\usepackage{enumerate}

\usepackage{cite}

\usepackage{color}
\definecolor{NoteColor}{rgb}{1,0,0}

\DeclareMathOperator{\Aut}{Aut}

\DeclareMathOperator{\Hom}{Hom}

\DeclareMathOperator{\Id}{Id}

\DeclareMathOperator{\Span}{Span}

\DeclareMathOperator{\Mat}{Mat}

\DeclareMathOperator{\Sp}{Sp}
\DeclareMathOperator{\GL}{GL}

\DeclareMathOperator{\OO}{O}

\DeclareMathOperator{\oo}{\mathfrak{o}}

\DeclareMathOperator{\spp}{\mathfrak{sp}}

\DeclareMathOperator{\Fix}{Fix}

\DeclareMathOperator{\Is}{Is}

\newcommand{\R}{\mathbb R}

\newcommand{\Z}{\mathbb Z}
\newcommand{\ZZ}{\mathbb Z}
\newcommand{\Q}{\mathbb Q}

\newcommand{\ab}{\mathrm{ab}}
\newcommand{\conj}{\mathrm{conj}}
\newcommand{\B}{\mathcal{B}}
\newcommand{\X}{\mathcal{X}}

\newcommand{\dbr}[1]{\{\!\!\{ #1 \}\!\!\}}
\newcommand{\Rep}{\mathcal{R}}

\theoremstyle{plain}
\newtheorem{teo}{Theorem}[section]

\newtheorem{cor}[teo]{Corollary}

\newtheorem{prop}[teo]{Proposition}

\theoremstyle{definition}
\newtheorem{df}[teo]{Definition}

\theoremstyle{remark}
\newtheorem{rem}[teo]{Remark}

\newcommand{\bs}{\setminus}
\newcommand{\defin}{\emph}

\usetikzlibrary{math}

\begin{document}
\title{On double brackets for marked surfaces}

\author[M. Gekhtman]{Michael Gekhtman} 
	
\author[E. Rogozinnikov]{Eugen Rogozinnikov}

\address{Michael Gekhtman \\ Depatrmant of Mathematics, University of Notre Dame, Notre Dame, USA}
\email{mgekhtma@nd.edu}

\address{Eugen Rogozinnikov \\ Max Planck Institute for Mathematics in the Sciences, Leipzig, Germany}
\email{erogozinnikov@gmail.com}



\begin{abstract}
    We propose a construction of a double quasi-Poisson bracket on the group algebra associated to the twisted fundamental group of a marked oriented surface $(S,P)$ with boundary, where $P$ is a finite set of marked points on the boundary of the surface $S$ such that on every boundary component there is at least one point of $P$. We show that this double bracket is a noncommutative generalization of the well-known Goldman bracket, defined on the space of free homotopy classes of loops on $S$. For an algebra $A$ without polynomial identities, we construct a double bracket on the space of decorated twisted $\GL_n(A)$-, symplectic and indefinite orthogonal local systems.
\end{abstract}

\maketitle
\tableofcontents
    \section{Introduction}

    The notion of a double Poisson algebra was introduced independently by M.~Van den Bergh~\cite{VdBergh04} and by W.~Crawley-Boevey, P.~Etingof, and V.~Ginzburg~\cite{CEG07}. A double Poisson bracket on a noncommutative algebra generalizes the classical Poisson bracket for the algebra of smooth functions on a manifold.

    Poisson structures on spaces of surface group representations into Lie groups were introduced by W.~Goldman in 1986 \cite{G3} and have since been studied by many researchers. In particular, they played a role in the development of cluster algebra theory, for which triangulations of surfaces and character varieties provide a rich source of important examples. For example, in~\cite{GSV05} it is shown how the notion of a Poisson bracket/2-form compatible with the cluster structure can be used in the case of Penner coordinates on decorated Teichm\"uller spaces to recover the celebrated Weyl--Petersson form. In much greater generality, V.~Fock and A.~Goncharov described in~\cite{FG} a Poisson structure on spaces of surface group representations into split Lie groups using their coordinates on the corresponding $\mathcal X$- and $\mathcal A$-cluster varieties.

    In~\cite{MT14}, G.~Massuyeau and V.~Turaev introduce a double quasi-Poisson bracket on the group algebra of the fundamental group $\pi_1(S,*)$ of a surface $S$ with nontrivial boundary with a base point on the boundary. They also proved that this double bracket gives rise to a quasi-Poisson structure on the representation space $\Rep(\pi_1(S,*),\GL_N(\R))=\Hom(\pi_1(S,*),\GL_N(\R))/\GL_N(\R)$. In his thesis~\cite{Art18}, S.~Arthamonov developed a categorical approach to double brackets and combined it with methods of~\cite{MT14} to study a noncommutative generalization of Goncharov--Kenyon integrable systems~\cite{GK13}.  In~\cite{AOS24}, this construction was utilized in the case of networks on surfaces and used to investigate integrability of noncommutative systems arising from weighted directed networks on a cylinder with noncommutative weights.
    

    In this article, we introduce a double quasi-Poisson bracket on the group algebra associated to the twisted fundamental group of a marked oriented surface $(S,P)$ with boundary, where $P$ is a finite set of marked points on the boundary of the surface $S$ such that on every boundary component there is at least one point of $P$. We show that this double bracket is a noncommutative generalization of the well-known Goldman bracket, defined on the space of free homotopy classes of loops on $S$. Moreover, after fixing an identification of the twisted fundamental group of $(S,P)$ with the product $\pi_1(S,P)\times\Z/2\Z$, this double bracket essentially agrees with the double bracket from~\cite{MT14} if $|P|=1$. Further, for an algebra $A$ without polynomial identities, we construct a double quasi-Poisson structure on the space of decorated twisted $A^\times$-local systems, which we think of as the space of equivalence classes of flat $A$-bundles over the unit tangent bundle $T'S$ of $S$ with the holonomy around fibers of $T'S\to S$ equal to $-1$, equipped additionally with parallel sections along paths surrounding every marked point of $P$ called (peripheral) decoration.
    
    Furthermore, we discuss the partial (non-)abelianization procedure, which was introduced and studied in~\cite{KR,K-Thesis,GK}. This procedure provides a map between the space of decorated twisted $\GL_n(A)$-local systems over $S$ and the space of decorated twisted $A^\times$-local systems over certain ramified $n:1$-covering $\Sigma$ of $S$. This construction is almost one-to-one, however it depends on certain choices. We show that the double bracket we introduced is essentially independent of these choices. This allows us to lift the double bracket from the space of decorated twisted $A^\times$-local systems over $\Sigma$ to the space of decorated twisted $\GL_n(A)$-local systems over $S$ which are transverse enough (for precise definition, we refer to Section~\ref{sec:abelianization}). Moreover, the compatibility with the partial (non-)abelianization allows us to equip the noncommutative cluster-like algebra $\mathcal A_n(S,P)$ introduced in~\cite{BR} in case $n=2$ and in~\cite{GK} in general case with a double quasi-Poisson structure.

    \medskip
    \noindent {\bf Structure of the paper:}
	In Section~\ref{sec:top_data}, we introduce the topological and combinatorial data needed for the definition of the double bracket, such as the twisted fundamental group of a punctured surface and a peripheral decoration. In Section~\ref{sec:double.bracket}, we revise the definition of the double (quasi-)Poisson bracket for associative algebras, construct several algebras associated to a decorated punctured surface and define a double quasi-Poisson bracket on them. In Section~\ref{sec:character_variety}, we introduce a space of decorated twisted $\GL_n(A)$-local systems over decorated punctured surfaces, where $A$ is an associative $\R$-algebra, and construct a double bracket on the space of decorated twisted $\GL_n(A)$-local systems. In Section~\ref{sec:symplectic}, we specify the double bracket introduced in Section~\ref{sec:double.bracket} for spaces of symplectic and indefinite orthogonal decorated twisted local systems over decorated punctured surfaces. 

    \medskip
    \noindent {\em Acknowledgements.}
     M.G. was partially supported by NSF grant DMS $\#$2100785. He is also grateful to the Max Planck Institute for Mathematics in the Sciences for its hospitality during the June and September 2024 research visits. E.R. was supported by a postdoc fellowship of the German Academic Exchange Service (DAAD) and funding from the European Research Council (ERC) under the European Union’s Horizon 2020 research and innovation programme (grant agreement No 101018839). The authors thank Arkady~Berenstein, Olivier~Guichard, Clarence~Kineider, Vladimir~Retakh, Michael~Shapiro, and Anna~Wienhard for helpful and interesting discussions about some aspects of this note.

\section{Topological and combinatorial data}\label{sec:top_data}
	
\subsection{Twisted fundamental group of a smooth surface}\label{sec:fund.group}

Let $S$ be an oriented smooth surface. Let $\pi\colon T'S\to S$ be the unit tangent bundle of $S$. Let $\tilde P\in T'S$ be a finite set such that for all $p,q\in\tilde P$, $\pi(p)\neq \pi(q)$ if $p\neq q$. Let $\Gamma(S,P)$ and $\Gamma(T'S,\tilde P)$ be the fundamental groupoids of $S$ relative to $P$, resp. of $T'S$ relative to $\tilde P$. For any $p\in P$, the fiber $\pi^{-1}(p)=T'_p S$ is diffeomorphic to a circle. Let $\delta_p$ be the generator of $\pi_1(T'_p S,\tilde p)$, which goes around $T'_p S$ once in the direction prescribed by the orientation of $S$. Let $\tilde e_p$ be the identity element of $\pi_1(T'_p S,\tilde p)$.
    
We define the group $\pi_1(T'S,\tilde P)$ as the group freely generated by $\Gamma(T'S,\tilde P)$ subject to the following additional relations: $\tilde e_{p}=1$ for all $p\in P$ and $\delta_{p_1}=\delta_{p_2}$ for all $p_1, p_2\in  P$, where $1$ is the identity element of the group freely generated by $\Gamma(T'S,\tilde P)$. We denote by $\pi_1(T'_\bullet S)$ the image of the group $\pi_1(T'_p S,\tilde p)$ inside $\pi_1(T'S,P)$ and by $\delta$ the image of the generator $\delta_p$. This image does not depend on the choice of $p\in P$. Moreover, $\pi_1(T_\bullet'S)$ is central in $\pi_1(T'S,P)$.

Similarly, we define the group $\pi_1(S,P)$ as the group freely generated by $\Gamma(S,P)$ subject to the following additional relations: $e_{p}=1$ for all $p\in P$, where $e_p$ is the identity element of the fundamental group $\pi_1(S,p)\subseteq\Gamma(S,P)$.

We now obtain the following central short exact sequence:
\begin{equation*}
    1\to \pi_1(T_{\bullet}'S) \to \pi_1(T'S,\tilde P) \to \pi_1(S,P) \to 1    
\end{equation*}
If $S$ is not closed, the group $\pi_1(S,P)$ is free, so the sequence above splits. The choice of a splitting corresponds to the choice of a non-vanishing vector field on $S$. Let $\pi^s_1(S,P)$ denote the quotient of $\pi_1(T'S,\tilde P)$ by the central subgroup generated by $\delta^2$:
$$\left< \delta^2 \right>\subset\left< \delta \right>= \pi_1(T_\bullet'S).$$
Thus, we obtain the short exact sequence:
\begin{equation}\label{eq:twisted.fund.grp}
	1\to \ZZ/2\ZZ\cong \left< \delta \right> / \left< \delta^2 \right> \to \pi^s_1(S,P) \to \pi_1(S,P) \to 1    
\end{equation}
that once again splits. Slightly abusing the notation, we denote the generator of the group $\left< \delta \right> / \left< \delta^2 \right>$ by $\delta$. Note that the sequence~\eqref{eq:twisted.fund.grp} also splits when $S$ is closed, since a closed surface of negative Euler characteristic always admits a vector field with zeroes of even indices only.

\subsection{Marked and punctured surfaces}
	
Let $\bar S$ be a compact orientable smooth surface of finite type with or without boundary. Let $P$ be a nonempty finite subset of $\bar S$ such that on every boundary component of $\bar S$ there is at least one element of $P$. Elements of $P$ are called \defin{marked points} of $\bar S$. A pair $(\bar S, P)$ is called a \defin{marked surface}. Sometimes, we will distinguish between elements of $P$ that lie in the interior of $\bar S$ -- \defin{internal marked points} and that lie on the boundary -- \defin{external marked points}. 

We also define a non-compact surface $S:=\bar S\bs P$. Non-compact surfaces that can be obtained in this way are called \defin{punctured surfaces}, with the exception of the (closed) disk with one or two punctures on the boundary and the sphere with one or two punctures. Elements of $P$ are called \defin{punctures} of $S$. Every punctured surface can be equipped with a complete hyperbolic structure of finite volume with totally geodesic boundary. For every such hyperbolic structure, all the internal punctures are cusps and all boundary curves are bi-infinite geodesics. 
	
An \defin{ideal polygon decomposition} of $S$ is a polygon decomposition with oriented edges of $\bar S$ whose set of vertices agrees with $P$, such that, if $\gamma$ is an edge of the polygon decomposition, then the opposite edge $\bar\gamma$ is also an edge of this polygon decomposition. All boundary curves are always edges of every polygon decomposition. We always consider edges of an ideal polygon decomposition as homotopy classes of oriented paths (relative to their endpoints) connecting points in $P$. Connected components of the compliment on $S$ to all edges of an ideal polygon decomposition $\mathcal T$ are called \defin{faces} or \defin{polygons} of $\mathcal T$. Every edge belongs to the boundary of one or two polygons. In the first case, an edge is called \defin{external}, in the second -- \defin{internal}. An ideal polygon decomposition is called an \defin{ideal triangulation} if all its faces are triangles. Any ideal polygon decomposition of $S$ can be represented by an ideal geodesic polygon decomposition as soon as a hyperbolic structure as above on $S$ is chosen.
	
\subsection{Peripheral decoration}
	
To consider framed and decorated twisted local systems, we will need the following additional data on the surface $S$:  for every internal puncture $p\in P$, we fix a neighborhood $S_p\subset S$ of $p\in P$ that is diffeomorphic to a punctured disk. For every external puncture $p\in P$, we choose a neighborhood $S_p\subset S$ of $p$ that is diffeomorphic to a punctured half-disk. We also assume that all $S_p$ are so small that they are pairwise disjoint.
	
Further, for every internal puncture $p\in P$ we fix a simple smooth loop $\beta_p\colon [0,1]\to S_p$ around $p$ such that $\dot\beta_p(0)=\dot\beta_p(1)$, oriented so that $p$ is on the left of $\beta_p$ according to the orientation of $S$. For every external puncture $p\in P$, we chose a simple smooth path $\beta_p\colon ([0,1],\{0,1\})\to (S_p, S_p\cap\partial S)$ connecting the boundary components separated by $p$ and tangent to them, once again with orientation given by the one on $S$.

In both cases, up to isotopy there is only one such $\beta_p$. Since all $\beta_p$ are smooth, we can lift them to the $T'S$ namely to the curve $[\beta_p(t),\dot\beta_p(t)]\in T'S$, $t\in[0,1]$. We denote this lift by $T'\beta_p\colon [0,1]\to T'S$. Notice that for every internal puncture $p$, the lift $T'\beta_p$ is always a loop.
	
If for every $p\in P$ a curve $\beta_p$ as above is chosen, then we say that the surface $S$ is \defin{decorated}, and the collection $\mathcal D=\{\beta_p\mid p\in P\}$ is called a \defin{decoration} of $S$. Elements of $\mathcal D$ are called \defin{decoration curves}. Further, for every $p\in P$, we choose $\tilde p'\in T'S$ which lies in the interior of the curve $T'\beta_p$ and denote by $\tilde P'\subset T'S$ the set of all $\tilde p'$ and by $P'$ the projection of $\tilde P'$ to $S$.

\begin{rem}
    If $S$ has no internal punctures and an ideal polygon decomposition of $S$ is chosen, then the set of its edges generates $\pi_1(T'S,\tilde P')$ in the following sense: every edge between $p$ and $q$ can be lifted to $T'S$ so that it intersects $T'\beta_p$ at $p'$ and $T'\beta_q$ at $q'$. Then the segment between $p'$ and $q'$ represents an element of $\pi_1(T'S,\tilde P')$. Notice that for two edges that meet at $p\in P$, such lifts may not only intersect at $p'$ but also in a small neighborhood of $p'$ (as on Figure~\ref{fig:triangle.lift}). The generating element $\delta$ of $\pi_1(T'_\bullet S)$ is represented by the concatenation of the lifts of two opposite edges. Moreover, every polygon decomposition provides a presentation of $\pi_1(T'S,\tilde P')$: the generating set of this presentation consists of all lifts of edges of the polygon decomposition, and the polygons provide the relations: for every polygon, the concatenation of boundary edges in $T'S$ of this polygon is equal to $\delta$ (as in Figure~\ref{fig:triangle.lift}).

    If $S$ has internal punctures, then in addition to edges of the chosen polygon decomposition, we need to add to the set of generators of $\pi_1(T'S,\tilde P')$ all the peripheral loops around internal punctures, which naturally represent elements of $\pi_1(T'S,\tilde P')$.
\end{rem}  

\begin{figure}[ht]
    \centering
    
    \begin{tikzpicture}
	\tikzmath{\i=0; \j=0; \p=\i*5; \q=\j*3;}
    \draw[bend right, -latex] (-2,0) to node [midway] (v1){} (2,-1);
	\draw[bend right, -latex] (-2,0) to node [midway,auto] {$\beta_p$} (2,-1);
	
	\node at (0.65,-2.55) {$\alpha_1$};
	\node at (-0.2,-2.5) {$\overline{\alpha}_1$};
	\node at (-2,-3.6) {$\alpha_2$};
	\node at (-1.7,-2.9) {$\overline\alpha_2$};
	\node at (-2,-1.4) {$\alpha_3$};
	\node at (-1.75,-2.1) {$\overline\alpha_3$};

    \draw[bend left, latex-] (-2,-5) to node [midway] (v2){} (2,-4);
	\draw[bend left,latex-] (-2,-5) to node [midway,below] {$\beta_q$} (2,-4);

    \draw[bend left, latex-] (-4,-0.5) to node [midway] (v3){} (-4,-4.5);
	\draw[bend left,latex-] (-4,-.5) to node [midway,left] {$\beta_r$} (-4,-4.5);

	\draw[latex-]  plot[smooth, tension=0.7] coordinates {(v2) (0.3,-3.3)  (0.3,-1.55) (v1)}; 
	\draw[]  plot[smooth, tension=0.7] coordinates {(v2) (-0.4,-3.5)  (-0.4,-1.55) (v1)}; 

	\draw[-latex]  plot[smooth, tension=0.7] coordinates {(v2) (-0.8,-3.8) (-2.9,-3) (v3)}; 
	\draw[]  plot[smooth, tension=0.7] coordinates {(v2) (-0.1,-3.5) (-2.7,-2.3) (v3)}; 

	\draw[latex-]  plot[smooth, tension=0.7] coordinates {(v1) (-0.8,-1.25) (-2.9,-2) (v3)}; 
	\draw[]  plot[smooth, tension=0.7] coordinates {(v1) (-0.1,-1.5) (-2.7,-2.7) (v3)};

    \end{tikzpicture}
    
    \caption{Lifts of edges of an ideal triangle. $\alpha_i\circ\overline\alpha_i=\delta$ for all $i\in\{1,2,3\}$, and $\alpha_3\alpha_2\alpha_1=\delta$.}
    \label{fig:triangle.lift}
\end{figure}
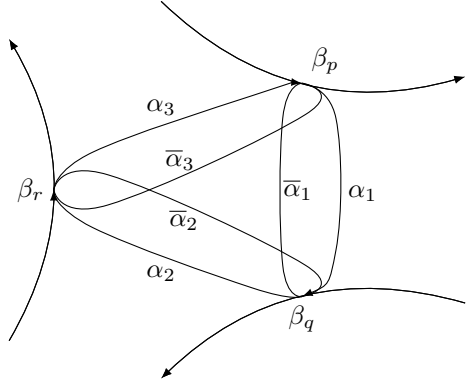

\section{Double bracket and surface algebras}\label{sec:double.bracket}

In this section, we recall the definition of a double (quasi-) Poisson bracket on an associative algebra, and define a double quasi-Poisson bracket on some algebra associated to a decorated punctured surface without internal punctures.

\subsection{Derivations and double brackets}


Let $A$ be a unital associative algebra over some field of characteristic zero. In what follows, if we write tensor powers of $A$, we always assume that the tensor product is taken over the center of $A$, and if we talk about linearity of maps between (tensor) powers of $A$, we always mean the linearity over the center of $A$.

Let $D\colon A\to A\otimes A$ be a linear map. $D$ is called a \defin{derivation for the outer bimodule structure on $A\otimes A$} if for all $a,b\in A$, $D(ab)=(a\otimes 1) D(b)+D(a)(1\otimes b)$. $D$ is called a \defin{derivation for the inner bimodule structure on $A\otimes A$} if for all $a,b\in A$, $D(ab)=(1\otimes a) D(b)+D(a)(b\otimes 1)$.

Let $n\in \mathbb N$ and let $\tau_n\colon A^{\otimes n}\to A^{\otimes n}$ be the cyclic permutation of factors, i.e. 
$$\tau_n(x_1\otimes\dots\otimes x_n)=x_2\otimes\dots\otimes x_n\otimes x_1.$$
A \defin{double bracket} on $A$ is a bilinear map $\dbr{\cdot,\cdot}\colon A\times A\to A\otimes A$ which is
\begin{itemize}
    \item a derivation in the second argument for the outer bimodule structure on $A\otimes A$, i.e. 
    $$\dbr{a,bc}=(b\otimes 1)\dbr{a,c}+\dbr{a,b}(1\otimes c);$$
    \item skew-symmetric, i.e. $\dbr{a,b}=-\tau_2(\dbr{b,a})$.
\end{itemize}
Since a double bracket is bilinear, it defines a linear map $\dbr{\cdot,\cdot}\colon A\otimes A\to A\otimes A$.
From the definition of the double bracket follows that it is a derivation in the first argument for the inner bimodule structure on $A\otimes A$, i.e. 
    $$\dbr{ab,c}=(1\otimes a)\dbr{b,c}+\dbr{a,c}(b\otimes 1)$$

\subsection{Triple brackets and double Poisson brackets} 

By~\cite[Proposition~2.3.1]{VdBergh04}, every double bracket defines a triple bracket, i.e. a map $\dbr{\cdot,\cdot,\cdot}\colon A^{\otimes 3}\to A^{\otimes 3}$ which satisfies the following two conditions: for all $a,b,c,d\in A$
$$\dbr{a,b,cd}=(c\otimes 1\otimes 1)\dbr{a,b,d}+\dbr{a,b,c} (1\otimes 1\otimes d)\;\text{ and }\;\dbr{c,a,b}=\tau_3\dbr{a,b,c}.$$

\begin{rem}
    To simplify the notation, we use Sweedler notations to omit the summation index in tensor products whenever it does not lead to any confusion, i.e. for $x=\sum_{i=1}^k x_i^{(1)}\otimes\dots\otimes x_i^{(k)}\in A^{\otimes n}$ we write $x=x^{(1)}\otimes\dots\otimes x^{(n)}$.
\end{rem}

For a derivation $D$, we can define a triple bracket associated to $D$ as follows: for $a,b,c\in A$
$$\dbr{a,b,c}_D:=\frac{1}{4}\left(D(a)^{(1)}D(b)^{(2)}\otimes D(b)^{(1)} D(c)^{(2)}\otimes D(c)^{(1)} D(a)^{(2)}\right).$$
For a double bracket $\dbr{\cdot,\cdot}$, the corresponding triple bracket is given by the formula:
$$\dbr{\cdot,\cdot,\cdot}:=\sum_{k=0}^{2} \tau_3^k\circ (\dbr{\cdot,\cdot}\otimes \Id)\circ(\Id\otimes\dbr{\cdot,\cdot})\circ\tau_3^{-k}.$$
A double bracket is \defin{Poisson} if the corresponding triple bracket vanishes, i.e. $\dbr{a,b,c}=0$ for all $a,b,c\in A$. 


\begin{rem}\label{rem:Lie bracket}
Every double Poisson bracket induces a Lie algebra structure on the \defin{cyclic space} associated to $A$ which is the abelianized space 
$$A^{\ab}:=A/[A,A] \text{ where } [A,A]=\Span_\Q\{ab-ba\mid a,b\in A\}.$$
Notice that since $[A,A]$ is just a subalgebra of $A$ and (in general) not an ideal, the cyclic space $A^{\ab}$ is a 
$\Q$-vector space without any natural $\Q$-algebra structure. The Lie algebra structure $\left<\cdot,\cdot\right>$ on $A^{\ab}$ induced by the double bracket $\dbr{\cdot,\cdot}$ is defined as follows:
$$\begin{array}{llcr}
\left<\cdot,\cdot\right>\colon & A^{\ab}\times A^{\ab} & \to & A^{\ab}\\
 & ([a],[b]) & \mapsto & [\mu(\dbr{a,b})]
\end{array}$$
where $a,b\in A$, $\mu\colon A\otimes A\to A$ is the multiplication map. This definition does not depend on representatives $a,b\in A$ of $[a],[b]\in A^{\ab}$ respectively. For more details, we refer to~\cite[Proposition~1.4]{VdBergh04} and~\cite[Section~4.5]{MT14}.    
\end{rem}

\subsection{Surface algebra and a double bracket on it}\label{sec:bracket}

Let now $(S,P)$ be a punctured surface without internal punctures with a decoration $\mathcal D=\{\beta_p\mid p\in P\}$. We consider the following quotient of the $\Q$-group algebra 
$$\mathcal A(S,P):=\Q[\pi_1^s(S,P')]/(\delta+1)=\Q[\pi_1(T'S,\tilde P')]/(\delta+1).$$
This algebra is unital. Moreover, since $\pi_1^s(S,P')$ is a central extension of a free group with $N\geq 2$ free generators by $\left< \delta\right>$, the center of $\mathcal A(S,P)$ is $\Q$. This algebra is of infinite dimension over $\Q$ if and only if $\pi_1^s(S,P')$ is infinite. We call $\mathcal A(S,P)$ the \defin{surface algebra}.

In this section, we define a double bracket $\dbr{\cdot,\cdot}$ on this algebra. First, we define a double bracket on the group algebra $\Q[\pi_1(T'S,\tilde P')]$. 
Let $\alpha_1$, $\alpha_2$ be two smooth simple curves that possibly intersect only along decoration curves of $S$. We decompose $\alpha_1$, $\alpha_2$ as products of curves that go from one  decoration curve to another one without intersecting any other decoration curve in between (such curves are, for example, edges of ideal polygon decompositions). We now assume that $\alpha_1$ and $\alpha_2$ are such curves. We define the double bracket for them and then extend the double bracket using the product rule. Let $\beta_p$ be a decoration curve where $\alpha_1$ and $\alpha_2$ meet. First we define the double bracket at $p$, we assume that $\alpha_1$ enters $\beta$ according to the orientation of $\beta_p$, $\alpha_2$ exits $\beta$ according to the orientation of $\beta_p$, and the first intersection point $\alpha_1\cap \beta_p$ is to the left of the first intersection point $\alpha_2\cap \beta_p$ (see Figure~\ref{fig:dbr_intersection}(a)). In this case, the concatenation of $\alpha_2\alpha_1$ along $\beta$ is well-defined, and we define: $\dbr{\alpha_1,\alpha_2}_p:=-\frac{1}{2}\alpha_2\alpha_1\otimes 1$. 
\begin{figure}[ht]
    \centering
\usetikzlibrary{math}

    \begin{tikzpicture}
    \foreach \i in {0,...,1}{
        \foreach \j in {0,...,1}{
            \tikzmath{\p=\i*5; \q=\j*3;}
            \draw[bend right, -latex] (-2+\p,0+\q) to node [midway] (v1){} (2+\p,0+\q);
            \draw[bend right, -latex] (-2+\p,0+\q) to node [midway,auto] {$\beta_p$} (2+\p,0+\q);
        
            \ifthenelse{\i = 0}
            {\draw[-latex]  plot[smooth, tension=0.7] coordinates {(-1.5+\p,-1.5+\q) (-1+\p,-0.8+\q) (v1)};}
            {\draw[latex-]  plot[smooth, tension=0.7] coordinates {(-1.5+\p,-1.5+\q) (-1+\p,-0.8+\q) (v1)};}
            \ifthenelse{\j = 1}
            {\draw[latex-]  plot[smooth, tension=0.7] coordinates {(1.5+\p,-1.5+\q) (1+\p,-0.8+\q) (v1)};}
		  {\draw[-latex]  plot[smooth, tension=0.7] coordinates     {(1.5+\p,-1.5+\q) (1+\p,-0.8+\q) (v1)};}

            \node at (-1.5+\p,-1+\q) {$\alpha_1$};
            \node at (1.5+\p,-1+\q) {$\alpha_2$};
        }}
        \node at (-2,3.5) {(a)};
        \node at (3,3.5) {(b)};
        \node at (-2,0.5) {(c)};
        \node at (3,0.5) {(d)};
    \end{tikzpicture}
    \caption{Intersection patterns}
    \label{fig:dbr_intersection}
\end{figure}
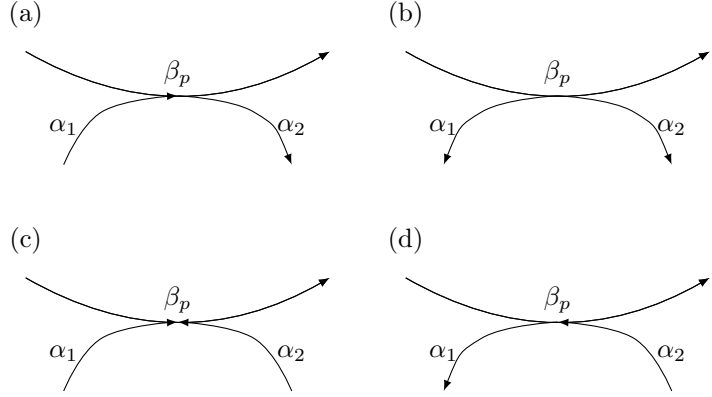

Using the rules:
$$\dbr{\alpha_1,\alpha_2^{-1}}_p=-(\alpha_2^{-1}\otimes 1)\dbr{\alpha_1,\alpha_2}_p(1\otimes\alpha_2^{-1}),$$
$$\dbr{\alpha_1^{-1},\alpha_2}_p=-(1\otimes\alpha_1^{-1})\dbr{\alpha_1,\alpha_2}_p(\alpha_1^{-1}\otimes 1),$$
$$\dbr{\alpha_2,\alpha_1}_p=-\tau_2(\dbr{\alpha_1,\alpha_2}_p),$$
we can extend the double bracket at $p$ to any kind of intersection of $\alpha_1$ and $\alpha_2$ along $\beta_p$ (see Figure~\ref{fig:dbr_intersection}):

(b) $\dbr{\alpha_1,\alpha_2}_p:=\frac{1}{2}\alpha_2\otimes \alpha_1$; (c) $\dbr{\alpha_1,\alpha_2}_p:=\frac{1}{2}\alpha_1\otimes \alpha_2$; (d) $\dbr{\alpha_1,\alpha_2}_p:=-\frac{1}{2}(1\otimes \alpha_1\alpha_2)$.

Finally, we define $\dbr{\alpha_1,\alpha_2}$ as the sum of all $\dbr{\alpha_1,\alpha_2}_p$ over all $p$ such that $\alpha_1$ and $\alpha_2$ meet along $\beta_p$. Using the product rule 
$$\dbr{\alpha_2\alpha_1,\alpha_3}=(1\otimes \alpha_2)\dbr{\alpha_1,\alpha_3}+\dbr{\alpha_2,\alpha_3}(\alpha_1\otimes 1),$$
the double bracket extends for all $\alpha_1$, $\alpha_2$ that are simple curves that possibly intersect only along decoration curves of $S$.

Let $\delta$ be the generator of $\pi_1(T_\bullet'S)$ from Section~\ref{sec:fund.group}, then $\dbr{\delta,\cdot}=0$ because $\delta$ is central in $\Q[\pi_1(T'S,\tilde P')]$. Smooth paths between points in $P'$ as above and $\delta$ generate $\pi_1(T'S,\tilde P')$, therefore, we obtain in this way a well-defined double bracket on $\Q[\pi_1(T'S,\tilde P')]$. Furthermore, since $\dbr{\delta+1,\cdot}=0$, this double bracket descends to the quotient algebra $\mathcal A(S,P)$. 



The bracket $\dbr{\cdot,\cdot}$ is in general not Poisson. However, it is so-called quasi-Poisson. To define this notion, we first introduce the \defin{uniderivation} $\partial_p\colon \mathcal A(S,P)\to\mathcal A(S,P)$ for every $p\in P$. For a~path $\alpha\colon [0,1]\to S$ which represents an element of $\pi_1^s(S,P')$, we define:
$$\partial_p(\alpha)=\begin{cases}
    \alpha\otimes 1- 1\otimes \alpha & \alpha(0),\alpha(1)\in\beta_p;\\
    \alpha\otimes 1 & \alpha(0)\in\beta_p,\;\alpha(1)\notin \beta_p;\\
    -1\otimes\alpha & \alpha(0)\notin \beta_p,\;\alpha(1)\in\beta_p;\\
    0  & \text{otherwise}.
\end{cases}$$
Let $\partial:=\sum_{p\in P}\partial_p$. A~double bracket is called \defin{quasi-Poisson} if it satisfies the following condition: $\dbr{\cdot,\cdot,\cdot}=\dbr{\cdot,\cdot,\cdot}_\partial$.



\begin{prop}
    The double bracket $\dbr{\cdot,\cdot}$ on $\mathcal A(S,P)$ defined in this section is quasi-Poisson. 
\end{prop}

\begin{proof}
    It is enough to check the quasi-Poisson property on generators at every point $p\in P$. We check it for the configuration of three paths as on Figure~\ref{fig:3arrows}. Other configurations can be checked similarly.

    \begin{figure}[ht]
    \centering
    \usetikzlibrary{math}
    \begin{tikzpicture}
            \tikzmath{\i=0; \j=0; \p=\i*5; \q=\j*3;}
            \draw[bend right, -latex] (-2+\p,0+\q) to node [midway] (v1){} (2+\p,0+\q);
            \draw[bend right, -latex] (-2+\p,0+\q) to node [midway,auto] {$\beta_p$} (2+\p,0+\q);
        
            \ifthenelse{\i = 0}
            {\draw[-latex]  plot[smooth, tension=0.7] coordinates {(-1.5+\p,-1.5+\q) (-1+\p,-0.8+\q) (v1)};}
            {\draw[latex-]  plot[smooth, tension=0.7] coordinates {(-1.5+\p,-1.5+\q) (-1+\p,-0.8+\q) (v1)};}
            \ifthenelse{\j = 1}
            {\draw[latex-]  plot[smooth, tension=0.7] coordinates {(1.5+\p,-1.5+\q) (1+\p,-0.8+\q) (v1)};}
		  {\draw[-latex]  plot[smooth, tension=0.7] coordinates     {(1.5+\p,-1.5+\q) (1+\p,-0.8+\q) (v1)};}

			\draw[latex-]  plot[smooth, tension=0.7] coordinates {(0.7,-1.5) (0.4,-0.8) (v1)}; 

            \node at (-1.6,-1.2) {$\alpha_1$};
            \node at (0.3,-1.2) {$\alpha_2$};
			\node at (1.6,-1.2) {$\alpha_3$};
    \end{tikzpicture}
    \caption{A configuration of three generators}
    \label{fig:3arrows}
\end{figure}
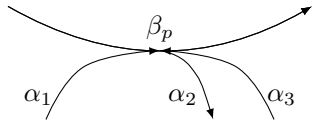
    $$\partial_p(\alpha_1)=-1\otimes\alpha_1,\;\partial_p(\alpha_2)=\alpha_2\otimes 1,\;\partial_p(\alpha_3)=-1\otimes\alpha_3.$$
    Therefore, $\dbr{\alpha_1\otimes\alpha_2\otimes\alpha_3}_\partial=\frac{1}{4}(1\otimes \alpha_2\alpha_3\otimes \alpha_1)$. Applying the triple bracket associated to $\dbr{\cdot,\cdot}$ at $p$, one can see that only one term is different from zero, namely the one corresponding to $k=2$. We obtain
    $$\dbr{\alpha_1\otimes\alpha_2\otimes\alpha_3}=\frac{1}{2}(\dbr{\cdot,\cdot}\otimes\Id)(\alpha_2\otimes\alpha_3\otimes\alpha_1)=\frac{1}{4}(1\otimes \alpha_2\alpha_3\otimes \alpha_1).$$
    Therefore, $\dbr{\cdot,\cdot}$ on $\mathcal A(S,P)$ is quasi-Poisson.
\end{proof}


\begin{rem}
    Similarly to Remark~\ref{rem:Lie bracket}, every double quasi-Poisson bracket on $\mathcal A(S,P)$ induces a Lie bracket on $\mathcal A(S,P)^{\ab}$.
\end{rem}

\begin{rem} 
    The double quasi-Poisson bracket as in this section can be also defined on the group algebra $\mathcal A_0(S,P)=\Q[\pi_1(S,P)]$ associated to the classical relative fundamental group without twisting. If, furthermore, $P=\{p\}$ is a one point set, this double bracket was defined and studied in~\cite{MT14}. In this case, the algebra $\mathcal A_0(S,p)$ is generated by all loops based at $p$. The abelianization $\mathcal A_0(S,p)^{\ab}:=\mathcal A_0(S,p)/[\mathcal A_0(S,p),\mathcal A_0(S,p)]$ is the $\Q$-vector space generated by free homotopy classes of loops on $S$ and, in fact, does not depend on $p$. Therefore, the induced bracket $\left< \cdot,\cdot \right>$ on $\mathcal A_0(S,p)^{\ab}$ is well-defined and agrees with the Goldman bracket ~\cite{G3}.
\end{rem}

\subsection{Double bracket and group action}

The double bracket defined in Section~\ref{sec:bracket} is naturally equivariant under diffeomorphisms of the surface $S$. More precisely, let $(S,P)$ be a punctured surface with a decoration $\mathcal D$. Let $f\colon S\to S$ be a diffeomorphism which preserves $P$ and $\mathcal D$, then $f$ extends to an automorphism $f_*$ of $\mathcal A(S,P)$ which maps $\alpha\in\pi_1^s(S,P')$ to $f\circ\alpha\in \pi_1^s(S,P')$. Further, we can extend $f_*$ to any $\mathcal A(S,P)^{\otimes n}$ componentwise.  Then for every two elements $\alpha_1,\alpha_2\in \pi_1^s(S,P')$, $\dbr{f\circ  \alpha_1,f\circ \alpha_2}=f_*\dbr{\alpha_1,\alpha_2}$. Notice that $f$ does not necessarily have to be orientation preserving. This implies the following corollary:

\begin{cor}
    Let $G$ be a group acting on a punctured surface $(S,P)$ by diffeomorphisms preserving the decoration $\mathcal D$. The double bracket $\dbr{\cdot, \cdot}$ is equivariant under this action. In particular, $\dbr{\cdot, \cdot}$ is equivariant under the action of the pure mapping class group of $S$.
\end{cor}

\section{Double bracket and the character variety}\label{sec:character_variety}

In this section, we introduce the space twisted representation of the twisted fundamental group of a punctured surface into $\GL_n(A)$ for an associative algebra $A$. This is a slight generalization of the classical space of surface group representations into Lie groups. We define a~framing and a~decoration for twisted representations, and then introduce a~double quasi-Poisson structure on the space of decorated twisted representations. 

\subsection{Twisted character varieties}
A homomorphism $\rho\colon\pi_1^s(S,P')\to \GL_n(A)$ is called \defin{twisted} if $\rho(\delta)=-\Id$ where $\delta$ is a loop homotopic to a fiber of the unit tangent bundle $T'S\to S$. We denote the space of all twisted homomorphisms as above by $\Hom^{tw}(\pi_1^s(S,P'), \GL_n(A))$. The group $\GL_n(A)$ acts on $\Hom^{tw}(\pi_1^s(S,P'), \GL_n(A))$ by conjugation. The quotient space 
$$\Rep^{tw}((S,P),\GL_n(A)):=\Hom^{tw}(\pi_1^s(S,P'), \GL_n(A))/\GL_n(A)$$
is called \defin{twisted $\GL_n(A)$-character variety}. Elements of $\Rep^{tw}((S,P),\GL_n(A))$ are called twisted representation. Notice that since $\delta=-1$ in $\mathcal A(S,P)$, 
$$\Rep^{tw}((S,P),\GL_n(A))=\Hom(\mathcal A(S,P), \Mat_n(A))/\GL_n(A).$$

By the Riemann-Hilbert correspondence, twisted representations are in one-to-one correspondence with gauge equivalence classes of $\GL_n(A)$-local systems over $T'S$ with the property that the holonomy around a fiber of $T'S\to S$ is $-1$. We call such a local system \defin{twisted}. 

A twisted local system $\mathcal L\to T'S$ is called \defin{framed} if for every $p\in P$, in the neighborhood $S_p$ of $p\in S$ which is homeomorphic to a half-disk with a marked point $p$ on its boundary flat subbundles $F_p^i$, $i\in\{0,\dots n\}$ are chosen such that for all $z\in S_p$ and all $i\in\{0,\dots n-1\}$, $F_p^i|_z\subset F_p^{i+1}|_z$ and $F_p^i|_z$ is isomorphic to $A^i$ as a right $A$-module and for $\ell^{i+1}_p:=F_p^{i+1}/F_p^{i}$, $\ell_p^{i+1}|_z$ is isomorphic to $A$ as a right $A$-module.

A twisted local system $\mathcal L\to T'S$ is called \defin{decorated} if it is framed and for every $p\in P$, flat regular sections $s_p^i$ of $\ell^{i}_p$, $i\in\{1,\dots, n\}$  along $T'\beta_p$ are chosen. Since we are dealing with surfaces without internal marked points, such flat sections always exist. We denote by $\Rep^{tw,d}((S,P),\GL_n(A))$ the space of equivalence classes of twisted decorated $\GL_n(A)$-local system.

Let now $n=1$, then $GL_1(A)=A^\times$. If we fix a decorated twisted $A^\times$-local system, then for every element of $\gamma\in\pi_1^s(S,P')$ connecting $p'$ and $q'$ an element $a_\gamma\in A^\times$ is well-defined as follows: $\mathcal P_\gamma(s(p'))=s_q(q')a_\gamma$, where $\mathcal P_\gamma(s(p'))$ is the parallel transport along $\gamma$ from $p'$ to $q'$ of the element $s_p(p')$. In fact, $a_\gamma$ does not depend on the choice of $p'$ and $q'$ on $T'\beta_p$ resp. $T'\beta_q$. Moreover, $a_\gamma$ is invariant under the action of the gauge group on decorated twisted local systems, so in fact $a_\gamma$ is a function on $\Rep^{tw,d}((S,P),A^\times)$. Since the algebra $\mathcal A(S,P)$ is generated by elements of $\pi_1^s(S,P')$, every $f\in \mathcal A(S,P)$ can be seen as a function on $\Rep^{tw,d}((S,P),A^\times)$ with values in $A$. We denote by $\mathcal F_A(S,P)$ the algebra of all such functions. The natural map 
\begin{equation}\label{eq:paths_to_functions}
\mathcal A(S,P)\to \mathcal F_A(S,P)    
\end{equation}
is a surjective homomorphism of algebras. 

However, the push-forward of the double quasi-Poisson bracket defined in the previous section does not always define a double quasi-Poisson bracket on the algebra $\mathcal F_A(S,P)$ when the homomorphism $\mathcal A(S,P)\to \mathcal F_A(S,P)$ is not injective. For example, the algebra $A=\Mat_n(K)$ for any field $K$ is finite dimensional and  has polynomial identities over $\Q$. This implies that the algebra $\mathcal F_A(S,P)$ has polynomial identities over $\Q$. However, these identities are not respected by the double bracket we introduced. If the map $\mathcal A(S,P)\to \mathcal F_A(S,P)$ is injective, then the push forward under this map induces a double quasi-Poisson structure on $\Rep^{tw}((S,P),A^\times)$. We call an algebra $A$ with this property \defin{big}.

\begin{rem}
    Algebras of bounded linear operators on an infinite dimensional separable Hilbert space provide examples of big algebras.
\end{rem}

\subsection{Double bracket and partial (non-)abelianization}\label{sec:abelianization}

Although in general the double bracket cannot be pushed forward to the character variety for Lie groups over finite dimensional algebras, it behaves nicely with respect to the partial non-abelianization described in~\cite{KR}. 

Let $S$ be a surface as before. The tautological twisted $\mathcal A(S,P)$-local system is the $\mathcal A(S,P)$-local system such that the holonomy along $\gamma\in\pi_1^s(S,P')$ is equal to $\gamma\in \mathcal A(S,P)$. Every representation $\mathcal A(S,P)\to A$ gives rise to a twisted $A^\times$-local system, i.e. to an element of $\Rep^{tw}((S,P),A^\times)$.



Let $\mathcal T$ be an ideal triangulation of $(S,P)$ and $A$ be a unital associative  $\Q$-algebra. A framed twisted local system $\mathcal L\in\Rep^{tw,d}((S,P),\GL_n(A)$ is called $\mathcal T$-transverse if for every ideal triangle with ideal vertices $p,q,r\in P$, the subbundles $F_p^i$, $F_q^j$, and $F_r^k$ are in direct sum for all $i+j+k\leq n$ (when the corresponding subbundles are transported parallelly to some common internal point of the triangle). 

In~\cite{GK}, a procedure is described that identifies the moduli space of twisted $\mathcal T$-transverse framed (or decorated) $\GL_n(A)$-local systems over a surface $S$ and the moduli space of twisted (resp. decorated) $A^\times$-local systems over a certain $n:1$-ramified covering $\Sigma$ of $S$ with only simple ramification points. This construction is equivalent to the partial non-abelianization introduced in~\cite{KR} (for $n=2$) and in~\cite{K-Thesis} (for general $n$) using spectral networks. This identification is a homomorphism of moduli spaces.

We sketch the construction of the covering $\Sigma$, for more details we refer the reader to~\cite{KR,GK,gmn13,hn16,KKRTT}. The construction is local: in the interior an ideal triangle $T\subseteq S$, we chose $\frac{n(n+1)}{2}$ points that we will call ``white'', and $\frac{n(n-1)}{2}$ ``black'' points that will be the ramification points of the covering $\Sigma\to S$. Further, we chose $n$ ``black'' points on every edge, but they will be regular (see Figure~\ref{fig:covering}).

\begin{figure}[ht]
    \centering
\usetikzlibrary{math}

\usetikzlibrary{calc}
    \begin{tikzpicture}[scale=1.5]

\node[circle, inner sep = 1.5pt](v3) at (0,2) {};
\node[circle, inner sep = 1.5pt](v2) at (2.5,-2) {};
\node[circle, inner sep = 1.5pt](v1) at (-2.5,-2) {};

\draw [] (v1)-- (v3);
\draw [] (v1)-- (v2);
\draw [] (v2)-- (v3);

\foreach \t in {0,...,4}{
	\foreach \s in {0,...,\t}{
	\node[fill, circle, inner sep = 1.5pt](v\t\s) at ($(v1)+(\t*5/4-\s*5/8,\s)$) {};
	}}

\foreach \t in {1,...,3}{
	\foreach \s in {1,...,\t}{
	\node[circle, draw=black, inner sep=0pt, minimum size=5pt](w\t\s) at ($(v1)+(\t*5/4-\s*5/8+0.5*5/4,\s-0.4)$) {};
	}}

\foreach \t in {1,...,3}{
	\foreach \s in {1,...,\t}{	\pgfmathtruncatemacro{\sn}{\s-1}
	\draw [red, semithick, bend left=10] (v\t\s) to node [auto]{} (w\t\s);
	\draw [red, semithick, bend left=10,-latex] (w\t\s) to node [auto]{} (v\t\sn);
}}

\foreach \t in {1,...,3}{
	\foreach \s in {1,...,\t}{	\pgfmathtruncatemacro{\tn}{\t+1}
	\draw [blue, semithick, bend right=10,latex-] (v\t\s) to node [auto]{} (w\t\s);
	\draw [blue, semithick, bend right=10] (w\t\s) to node [auto]{} (v\tn\s);
}}

\foreach \t in {1,...,3}{
	\foreach \s in {1,...,\t}{	\pgfmathtruncatemacro{\tn}{\t+1}
	\pgfmathtruncatemacro{\sn}{\s-1}
	\draw [green, thick, bend right=10,latex-] (v\tn\s) to node [auto]{} (w\t\s);
	\draw [green, thick, bend right=10] (w\t\s) to node [auto]{} (v\t\sn);
}}

\end{tikzpicture}
    \caption{Construction of the covering $\Sigma$ for $n=3$}
    \label{fig:covering}
\end{figure}
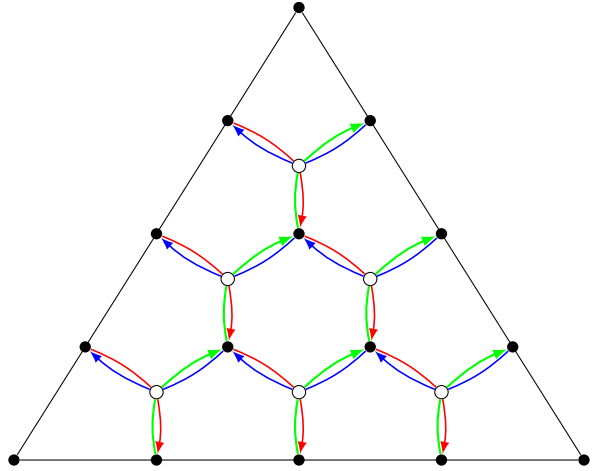

Further, we connect black and white points by $3n$ oriented \defin{zigzag paths} that start and end at black points on edges following the rule that the path turns right at white points and left at black points (see red, blue and green paths on Figure~\ref{fig:covering}). Every segment of a zigzag path connecting a white and a black point and not containing any black and white points in its interior is called a \defin{simple segment}. Every zigzag path defines a polygonal region that lies to the right of the path called a \defin{cell}. Every cell contains exactly one of three marked point of $T$.

For every simple segment, there exist exactly two cells sharing it. We glue the cells along these common simple segments and obtain the covering $\hat T$ of the initial triangle $T$. It is easy to see then that white points are regular because at white points three cells with angles $2\pi/3$ meet, but at the black points in the interior of the triangle three cells with angles $4\pi/3$ meet, i.e. these points are ramification points of index~$2$. 

If the triangle $T$ shares an edge with another triangle $T'$, then every zigzag path of $T$ meeting this edge continues to a zigzag path in $T'$ and vise versa. So for the cell of $T$ which corresponds to this zigzag path, there is a cell in $T'$ corresponding to its continuation. This gives a rule how to glue lifts $\hat T$ and $\hat T'$ together along the common boundary. Applying this construction consequently to all triangles, we obtain an $n:1$-ramified covering $\Sigma\to S$ with only simple ramifications. 

In the lift $\hat T\subseteq \Sigma$ of a triangle $T\subseteq S$, for every edge of a cell, there is another cell that is glued with this cell along this edge. So there is a unique path that goes from  the marked point in the first cell to the marked point in the second cell, intersecting this common edge they were glued along. The collection of all such paths form a generating set of $\pi_1(\Sigma,\hat P)$ where $\hat P$ is the full lift of $P$ to $\Sigma$. We also denote by $\hat P'\subset\Sigma$ the full lift of $P'\subset S$ to $\Sigma$.

If we do a flip in the triangulation $\mathcal T$, i.e. replace one diagonal by the other one in a quadrilateral, we can partially abelianize our $\GL_n(\mathcal A(\Sigma,\hat P))$-local system on $S$ and obtain a new twisted $\mathcal A(\Sigma,\hat P)$-local system on $\Sigma$. Since a changing of an ideal triangulation on $S$ induces not only a change of the ideal polygon decomposition on $\Sigma$ but also the local system changes, we need to prove that the double bracket is equivariant under this change. In other words, a flip $F$ induces the algebra homomorphism $F_*\colon \mathcal A(\Sigma,\hat P)\to \mathcal A(\Sigma,\hat P)$. We need to check that for every $x_1,x_2\in \mathcal A(\Sigma,\hat P)$, $F_*(\dbr{x_1,x_2})=\dbr{F_*(x_1),F_*(x_2)}$.  

A flip can be realized as a sequence of elementary transformation in the polygon decompositions of $S$ and $\Sigma$. Every elementary transformation happens in a quadrilateral (in $S$) with two ramification points connecting the same sheets (in $\Sigma$) as on Figures~\ref{fig:flipS} and~\ref{fig:flipSigma} (non-ramified sheets are not affected).


\begin{figure}[ht]
    \centering
    \begin{tikzpicture}

\draw [] (0,-2) node [label= below:$1$] {} --  (3,0) node[fill, circle, inner sep = 1.5pt,label=right:$2$] (v2) {};
\draw [] (v2) -- (0,2) node [fill, circle, inner sep = 1.5pt,label= above:$3$] (vsm) {};
\draw [] (0,2) -- (-3,0) node [fill, circle, inner sep = 1.5pt,label= left:$4$] (vsp){};
\draw [] (-3,0) --  (0,-2) node [fill, circle, inner sep = 1.5pt] {};

\node[fill, circle, inner sep = 1.5pt](vs) at (0,2) {};
\node[fill, circle, inner sep = 1.5pt](v1) at (0,-2) {};

\draw [] (v1)-- (vs);
\draw [red,densely dashed](v2)--(vsp);


\node at (0,-3){};

\end{tikzpicture}
    \caption{Flip in $Q\subset S$}
    \label{fig:flipS}
\end{figure}
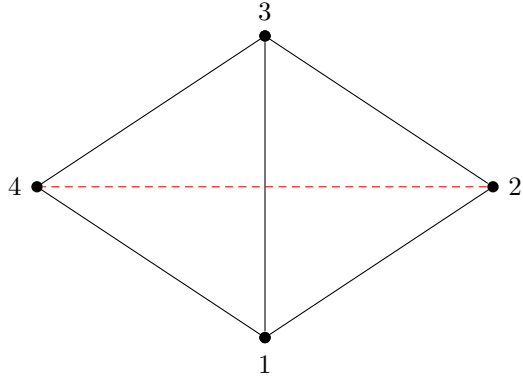

We do a flip in a quadrilateral $Q$ as shown on Figure~\ref{fig:flipS}. The lift of this quadrilateral is the cylinder $\hat Q$ (see Figure~\ref{fig:flipSigma}). We denote by $a_{ij}\in\pi_1^s(\Sigma,\hat P')$ an edge of the ideal hexagonal decomposition of $\Sigma$ going from the vertex $j'$ to the vertex $i''$. Then $F_*(a_{24})=a_{23} a^{-1}_{13} a_{14} +a_{21} a^{-1}_{31} a_{34}$. If $x$ is an edge of $\hat{\mathcal T}$ which is not changed by $F$, i.e. it does not intersect the cylinder $\hat Q$, then $F_*(x)=x$ and $\dbr{F_*(x),F_*(a_{24})}=\dbr{x,F_*(a_{24})}=F_*(\dbr{x,a_{24}})$. 

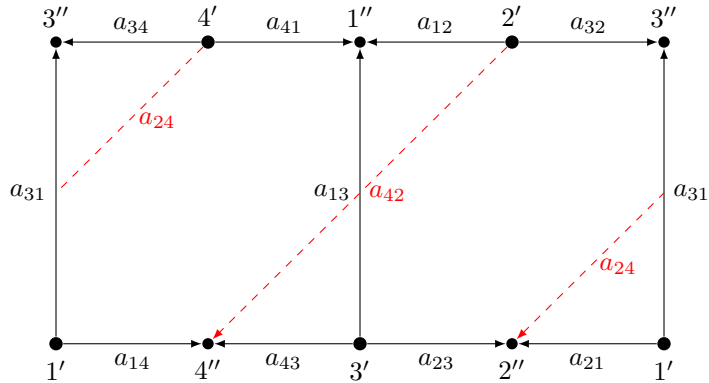
\begin{figure}[ht]
    \centering
    \begin{tikzpicture}

\node [circle, fill, inner sep=0pt, minimum size=5pt,label=below:$1'$] (vw1) at (5,-2){};
\node [fill, circle, inner sep = 1.5pt,label=above:$3''$] (vb3) at (5,2){};

\node [circle, fill, inner sep=0pt, minimum size=5pt,label=above:$4'$] (vw4) at (7,2){};
\node [fill, circle, inner sep = 1.5pt,label=below:$4''$] (vb4) at (7,-2){};

\node [circle, fill, inner sep=0pt, minimum size=5pt,label=below:$3'$] (vw3) at (9,-2){};
\node [fill, circle, inner sep = 1.5pt,label=above:$1''$] (vb1) at (9,2){};

\node [circle, fill, inner sep=0pt, minimum size=5pt,label=above:$2'$] (vw2) at (11,2){};
\node [fill, circle, inner sep = 1.5pt,label=below:$2''$] (vb2) at (11,-2){};

\node [circle, fill, inner sep=0pt, minimum size=5pt,label=below:$1'$] (vw11) at (13,-2){};
\node [fill, circle, inner sep = 1.5pt,label=above:$3''$] (vb33) at (13,2){};

\draw [latex-](vb1) -- node[left]{$a_{13}$} (vw3);
\draw [-latex](vw1) -- node[left]{$a_{31}$}  (vb3);
\draw [-latex](vw11) -- node[right]{$a_{31}$}  (vb33);

\draw [-latex](vw2) --  node[above]{$a_{32}$} (vb33);
\draw [-latex](vw2) --  node[above]{$a_{12}$} (vb1);

\draw [-latex](vw3) --  node[below]{$a_{23}$} (vb2);
\draw [-latex](vw11) --  node[below]{$a_{21}$} (vb2);

\draw [latex-](vb4) --  node[below]{$a_{14}$} (vw1);
\draw [-latex](vw4) --  node[above]{$a_{41}$} (vb1);

\draw [latex-](vb3) --  node[above]{$a_{34}$} (vw4);
\draw [latex-](vb4) --  node[below]{$a_{43}$} (vw3);

\draw [red,dashed,-latex](vw2) -- node[right]{$a_{42}$}(vb4);

\draw [red,dashed](vw4) -- node[right]{$a_{24}$}(5,0);
\draw [red,dashed,latex-](vb2) -- node[right]{$a_{24}$}(13,0);
\end{tikzpicture}
    \caption{Flip in $\hat Q\subset \Sigma$}
    \label{fig:flipSigma}
\end{figure}

Only two edges intersect $\hat Q$, namely $a_{13}$ and $a_{31}$. We check the equivariance for $a_{31}$. The calculation for $a_{13}$ is similar. $F_*(a_{31})=a_{34}F_*(a_{24}^{-1})a_{21}$. The direct computation of the double bracket shows the equivariance.

This means that the double bracket can be lifted using the non-abelianization map to the moduli space $\Rep^{tw,d}((S,P),\GL_n(A))$ of decorated twisted $\GL_n(A)$-local systems on $S$ that are transverse to at least one ideal triangulation of $S$ for a big algebra $A$.

\subsection{Relation to noncommutative cluster algebras}

The (non-)abelianization allows us to construct a $\mathbb Q$-algebra $\mathcal A_n(S,P)$ as follows: As we have seen, every ideal triangulation $\mathcal T$ of $S$ gives rise to a generating system of the fundamental group $\pi_1^s(\Sigma,\hat P')$, which we denote by $\Delta_\mathcal T$. Every flip relating two triangulations $\mathcal T$ and $\mathcal T'$ provides a change of a generating system from $\Delta_\mathcal T$ to $\Delta_{\mathcal{T}'}$ together with formulas relating generators of  $\Delta_\mathcal T$ and $\Delta_{\mathcal{T'}}$ which are noncommutative Laurent polynomials and are usually called \defin{mutation formulas}.

The generating set of the algebra $\mathcal A_n(S,P)$ is defined to be the union $\bigcup_\mathcal T \Delta_\mathcal T$. The relations in $\mathcal A_n(S,P)$ are prescribed by the following two rules:
\begin{itemize}
    \item For every $\mathcal T$, the set $\Delta_\mathcal T$ generates the group $\pi_1^s(\Sigma,\hat P')$. All the relations that the elements of $\Delta_\mathcal T$ satisfy as generators of $\pi_1^s(\Sigma,\hat P')$ are relations in $\mathcal A_n(S,P)$. 
    \item If $\mathcal T$ and $\mathcal T'$ are related by a flip, all the mutation formulas between $\Delta_\mathcal T$ and $\Delta_{\mathcal T'}$ are relations in $\mathcal A_n(S,P)$.
\end{itemize}

Following the approach introduced by Berenstein and Retakh (cf.~\cite{BR,BHR,BHRR}), we call the algebra  $\mathcal A_n(S,P)$ the \defin{noncommutative cluster-like algebra} associated to the marked surface $(S,P)$.

Since the double bracket on $\mathcal A(\Sigma,\hat P)$ is equivariant with respect to flips, the algebra $\mathcal A_n(S,P)$ can be equipped with a double bracket as follows: for every ideal triangulation $\mathcal T$, we define the double bracket on edges of $\mathcal T$ as it is done in $\mathcal A(\Sigma,\hat P)$. The equivariance by flips guarantees that this bracket extends to the entire $\mathcal A_n(S,P)$. Therefore, $\mathcal A_n(S,P)$ becomes naturally a double quasi-Poisson algebra.

\section{Symplectic and indefinite orthogonal local systems}\label{sec:symplectic}

	Involutive algebras form an important class of noncommutative algebras. Over involutive algebras, generalizations of many classical groups can be constructed (e.g. orthogonal groups, symplectic groups). In this chapter, we define algebras with anti-involutions and symplectic and indefinite orthogonal groups over such algebras that were introduced and studied in~\cite{ABRRW}. Further, following~\cite[Section~5.3]{KR}, we introduce decorated twisted symplectic and indefinite orthogonal local systems and provide a double quasi-Poisson bracket on the moduli spaces of decorated twisted symplectic and indefinite orthogonal local systems for big involutive algebras.
	
	\subsection{Symplectic and indefinite orthogonal groups over involutive algebras}
	
	Let $A$ be a unital associative, possibly noncommutative 
	$\Q$-algebra.
	
	\begin{df}\label{def:antiinv}
		An \defin{anti-involution} on $A$ is a $\Q$-linear map $\sigma\colon A\to A$ such that
		\begin{itemize}
			\item $\sigma(ab)=\sigma(b)\sigma(a)$;
			\item $\sigma^2=\Id$.
		\end{itemize}
		An \defin{involutive $\Q$-algebra} is a pair $(A,\sigma)$, where $A$ is a $\Q$-algebra and $\sigma$ is an anti-involution on $A$.
	\end{df}
	
	\begin{df}
		Two elements $a,a'\in A$ are called \defin{congruent}, if there exists $b\in A^\times$ such that $a'=\sigma(b)ab$.
	\end{df}
	
	\begin{df} An element $a\in A$ is called \defin{$\sigma$-symmetric} if $\sigma(a)=a$. An element $a\in A$ is called \defin{$\sigma$-anti-symmetric} if $\sigma(a)=-a$. We denote
		$$A^{\sigma}:=\Fix_A(\sigma)=\{a\in A\mid \sigma(a)=a\},$$
		$$A^{-\sigma}:=\Fix_A(-\sigma)=\{a\in A\mid \sigma(a)=-a\}.$$
	\end{df}
	
	\begin{df}
		The closed subgroup $$U_{(A,\sigma)}=\{a\in A^\times\mid \sigma(a)a=1\}$$ of $A^\times$ is called the \defin{unitary group} of $A$. The Lie algebra of $U_{(A,\sigma)}$ agrees with $A^{-\sigma}$.
	\end{df}

	\begin{df}
		The \defin{symplectic group} $\Sp_2$ over $(A,\sigma)$ is 
		$$\Sp_2(A,\sigma):=\{g\in \Mat_2(A)\mid \sigma(g)^t\Omega g=\Omega\}\text{ where } \Omega=\begin{pmatrix}
        0 & 1 \\
        -1 & 0
      \end{pmatrix}.$$
      The \defin{indefinite orthogonal group} $\OO_{(1,1)}$ over $(A,\sigma)$ is 
		$$\OO_{(1,1)}(A,\sigma):=\{g\in \Mat_2(A)\mid \sigma(g)^t\Omega g=\Omega\}\text{ where } \Omega=\begin{pmatrix}
        0 & 1 \\
        1 & 0
      \end{pmatrix}.$$
      
	\end{df}
	We can determine the Lie algebras $\spp_2(A,\sigma)$ of $\Sp_2(A,\sigma)$ and $\oo_{(1,1)}(A,\sigma)$ of $\OO_{(1,1)}(A,\sigma)$:
	\begin{align*}
	    \spp_2(A,\sigma)& =\left\{g\in \Mat_2(A)\mid \sigma(g)^t\begin{pmatrix}
        0 & 1 \\
        -1 & 0
      \end{pmatrix}+ \begin{pmatrix}
        0 & 1 \\
        -1 & 0
      \end{pmatrix} g=0\right\}\\
    & =\left\{\begin{pmatrix}
		x & z \\
		y & -\sigma(x)
	\end{pmatrix}\mid x\in A,\;y,z\in A^\sigma\right\},
	\end{align*}
    \begin{align*}
        \oo_{(1,1)}(A,\sigma)& =\left\{g\in \Mat_2(A)\mid \sigma(g)^t\begin{pmatrix}
        0 & 1 \\
        1 & 0
      \end{pmatrix}+ \begin{pmatrix}
        0 & 1 \\
        1 & 0
      \end{pmatrix} g=0\right\}\\
      & =\left\{\begin{pmatrix}
		x & z \\
		y & -\sigma(x)
	\end{pmatrix}\mid x\in A,\;y,z\in A^{-\sigma}\right\}.
    \end{align*}
    We now consider $A^2$ as a right $A$-module over $A$. Let $x,y\in A^2$, the $A$-valued form $$\omega(x,y):=\sigma(x)^t\Omega y$$
    is called  \defin{standard symplectic} if $\Omega=\begin{pmatrix}
        0 & 1 \\
        -1 & 0
      \end{pmatrix}$ and \defin{standard indefinite orthogonal} if $\Omega=\begin{pmatrix}
        0 & 1 \\
        1 & 0
      \end{pmatrix}$. 
    Both forms have many properties in common. Therefore, below we discuss both cases in parallel. When some property is specific only for one of these forms, we will emphasize this explicitly.
      
    Let $(x,y)$ be a basis of $A^2$, i.e. a pair of elements $x,y\in A^2$ such that the map $A^2\to A^2$, $(a,b)\mapsto xa+yb$ for $a,b\in A^2$ is an isomorphism of right $A$-modules. We say that $x\in A^2$ is \defin{regular} if there exist $y\in A^2$ such that $(x,y)$ is a basis. We say that a basis $(x,y)$ is \defin{isotropic} if $\omega(x,x)=\omega(y,y)=0$. We say that this basis is \defin{normalized} if, furthermore $\omega(x,y)=1$.
	
	Let $x\in A^2$ be a regular isotropic element, i.e. $\omega(x,x)=0$. We call the set $xA:=\{xa\mid a\in A\}$ an \defin{isotropic $A$-line spanned by $x$}. The space of all isotropic $A$-lines is denoted by $\Is(\omega)$.
	
	\subsection{Symplectic and indefinite orthogonal local systems}
	
	We denote by $G=\Aut(\omega)$ which is $\Sp_2(A,\sigma)$ if $\Omega$ is standard symplectic and $\OO_{(1,1)}(A,\sigma)$ if $\Omega$ is standard indefinite orthogonal.
 
    Let $\mathcal L\to T'S$ be a twisted $\GL_2(A)$-local system  over $S$. We say that $\mathcal L$ is a \defin{twisted $G$-local system} (or just \defin{twisted symplectic resp. orthogonal indefinite local system}) if there exists a parallel field of the standard symplectic, resp. the standard orthogonal indefinite 2-form $\omega\colon\mathcal L\times\mathcal L\to A$ on $T'S$. 
	
	A framing of a twisted $G$-local system is called \defin{isotropic} if for all $p\in P$ a parallel subbundle $F^1_p$ in the neighborhood $S_p$ defining the framing is isotropic with respect to the field of the form $\omega$. A decoration $(s_p^1,s_p^2)_{p\in P}$, $s^1_p\in F^1_p=\ell^1_p$, $s^2_p\in \mathcal L|_{S_p} / F^1_p=\ell^2_p$ of a twisted  $G$-local system is called \defin{normalized} if $\omega(s_p^1,s_p^1)=0$ and $\omega(s_p^1,s_p^2)=1$ for all $p\in P$ (cf.~ Section~\ref{sec:character_variety} in case $n=2$).
	
	\begin{rem} 
	    Note, that if $\omega(s_p^1,s_p^1)=0$, then the expression $\omega(s_p^1,s_p^2)$ is well-defined. Indeed, let $\tilde s_p^2$ and $(\tilde s^2_p)'$ be two lifts of $s_p^2$ to $A^2$. Then $(\tilde s^2_p)'=\tilde s^2_p+s^1_p a$ for some $a\in A$. Furthermore, 
	    $$\omega(s^1_p,(\tilde s^2_p)')=\omega(s^1_p,\tilde s^2_p+s^1_pa)=\omega(s^1_p,\tilde s^2_p)=:\omega(s^1_p,s^2_p).$$
	    For a normalized decoration, it is always enough to choose $s^1_p$ for every $p\in P$. Then $s^2_p$ becomes uniquely defined by the normalization condition $\omega(s_p^1,s_p^2)=1$.
	\end{rem}
	
	A \defin{framed twisted  $G$-local system} is a twisted  $G$-local system with an isotropic framing. A~\defin{decorated twisted  $G$-local system} is a twisted  $G$-local system with a normalized decoration.
	
	\begin{rem}
		Note, that since $\omega$ is a parallel form of even degree, the parallel transport of $\omega$ around the fiber of $T'S$ is trivial.
	\end{rem}

\subsection{Double bracket for symplectic and indefinite orthogonal local systems}

As we have seen in Section~\ref{sec:abelianization}, for a marked surface $(S,P)$ with a decoration $\mathcal D$ and with an ideal triangulation $\mathcal T$, the partial abelianization procedure produces a marked surface $(\Sigma, \hat P)$ with a decoration $\hat{\mathcal{D}}$ and with a hexagonal decomposition $\hat{\mathcal T}$ which is a ramified double covering $\Sigma\to S$. Such a covering always admits the covering involution $\theta\colon\Sigma\to\Sigma$, which preserves the set $\hat P$, the decoration $\hat{\mathcal D}$ and the hexagonal decomposition $\hat{\mathcal T}$.

This involution extends naturally in two ways to anti-involutions $\theta^{\pm}$ on $\mathcal A(\Sigma,\hat P)$ as follows: for any edge $\gamma$ of the hexagonal decomposition $\hat{\mathcal T}$, we define $\theta^{\pm}(\gamma):=\pm(\theta\circ\gamma)^{-1}$. Then $\theta^{\pm}$ extend $\Q$-linearly as anti-involutions to $\mathcal A(\Sigma,\hat P)$. Further, we extend $\theta^{\pm}$ to $\mathcal A(\Sigma,\hat P)\otimes \mathcal A(\Sigma,\hat P)$ by the rule $\theta^\pm(x\otimes y):=\theta^\pm(y)\otimes \theta^\pm(x)$. This definition implies that the double bracket $\dbr{\cdot,\cdot}$ on $\Sigma$ is equivariant under $\theta^\pm$.

Following the non-abelianization procedure, the tautological local system on $\Sigma$ gives rise to a framed twisted $\GL_2(\mathcal A(\Sigma,\hat P))$-local systems on $S$. By the characterization of decorated twisted symplectic local systems from~\cite[Section~5.3]{KR} and similar characterization of decorated twisted indefinite orthogonal local systems, it follows that the non-abelianization of tautological local system on $\Sigma$ is in fact symplectic with respect to the anti-involution $\theta^+$ and indefinite orthogonal with respect to the anti-involution $\theta^-$. 

We call an involutive algebra $(A,\sigma)$ \defin{big} if $A$ is big and $\sigma$ is not algebraic, i.e. there is no polynomial $P$ over $A$ such that $\sigma(a)=P(a)$ for all $a\in A$. Then, similarly to the case of framed twisted $\GL_2(A)$-local systems, a double quasi-Poisson bracket can by induced on the character variety $\Rep^{tw,d}((S,P),G)$ for every big involutive algebra $(A,\sigma)$ where $G$ is either $\Sp_2(A,\sigma)$ or $\OO_{(1,1)}(A,\sigma)$ as the push-forward bracket under the bijective algebra homomorphism~\eqref{eq:paths_to_functions}.

\begin{rem}
    The condition to be big for an involutive algebra $(A,\sigma)$ is crucial. Otherwise, the double bracket is not well-defined on $\Rep^{tw,d}((S,P),G)$. For instance, let $A$ be abelian and $\sigma=\Id$. Then the elements of $\mathcal F_A(\Sigma,\hat P)$ defined by an element $\gamma\in\pi_1^s(\Sigma,\hat P)$ and $\theta\circ\gamma^{-1}\in\pi_1^s(\Sigma,\hat P)$ coincide. However, the double bracket with these two elements are in general different: on~Figure~\ref{fig:covering} for $\gamma=a_{14}$, $\dbr{a_{14},a_{43}}=\frac{1}{2}a_{14}\otimes a_{43}$, but $\theta\circ\gamma^{-1}=a_{41}$ and $\dbr{a_{41},a_{43}}=0$.
\end{rem}

\bibliographystyle{plain}
\bibliography{bibl}
\end{document}